\documentclass[12pt]{article}
\usepackage{a4}
\usepackage{amsthm}
\usepackage{amsmath}
\usepackage{amssymb}
\usepackage{amsfonts}
\usepackage{cite}
\usepackage{epsfig}
\usepackage{authblk}

\newtheorem{theorem}{Theorem}
\newtheorem{problem}{Problem}
\newtheorem{proposition}[theorem]{Proposition}
\newtheorem{lemma}[theorem]{Lemma}
\newtheorem{corollary}[theorem]{Corollary}
\newcommand{\II}{{\mathcal I}}
\newcommand{\NN}{\mathbb{N}}
\newcommand{\ZZ}{\mathbb{Z}}
\newcommand{\RR}{\mathbb{R}}
\newcommand{\TT}{\mathbb{T}}
\newcommand{\dd}{\;\mathrm{d}}

\makeatletter
 \def\blfootnote{\xdef\@thefnmark{}\@footnotetext} 
\makeatother

\begin{document}
\title{Coloring graphs by translates in the circle\blfootnote{The first author was supported by the Spanish Ministerio de Econom{\'i}a y Competitividad project MTM2017-83496-P. The third, fourth and fifth authors were supported by the MUNI Award in Science and Humanities of the Grant Agency of Masaryk University. The sixth author was supported by the German Research Foundation under Germany's Excellence Strategy - MATH+ (EXC-2046/1, project ID: 390685689). The seventh author was supported by project 18-13685Y of the Czech Science Foundation (GA\v{C}R).}}
		
\author{Pablo Candela,\; Carlos Catal\'a,\; Robert Hancock,\; Adam Kabela, Daniel Kr\'al',\; Ander Lamaison,\; Llu\'is Vena}

\date{}

\newcommand{\Addresses}{{
  \bigskip
  \footnotesize

\noindent  P.~Candela, \textsc{Universidad Aut\'onoma de Madrid, and ICMAT, Madrid 28049, Spain}\par\nopagebreak
  \textit{E-mail}: \texttt{pablo.candela@uam.es}

  \medskip

\noindent  C.~Catal\'a, \textsc{Universidad Aut\'onoma de Madrid, Madrid 28049, Spain}\par\nopagebreak
  \textit{E-mail}: \texttt{ca.cataladlt@gmail.com}

  \medskip

\noindent  R.~Hancock, \textsc{Institut f\"ur Informatik, University of Heidelberg, Im Neuenheimer Feld 205, 69120, Heidelberg, Germany} \\
  \indent\qquad\textsc{Previous affiliation: Faculty of Informatics, Masaryk University, Bota\-nick\'a 68A, 602 00 Brno, Czech Republic }\par\nopagebreak
  \textit{E-mail}: \texttt{hancock@informatik.uni-heidelberg.de}
  
  \medskip
  
\noindent  A.~Kabela, \textsc{Faculty of Informatics, Masaryk University, Botanick\'a 68A, 602 00 Brno, Czech Republic}\par\nopagebreak
  \textit{E-mail}: \texttt{kabela@fi.muni.cz}
  
  \medskip

\noindent  D.~Kr\'al', \textsc{Faculty of Informatics, Masaryk University, Botanick\'a 68A, 602 00 Brno, Czech Republic} \\
  \indent\qquad\textsc{Previous affiliation: Mathematics Institute, DIMAP and Department of Computer Science, University of Warwick, Coventry CV4 7AL, UK}\par\nopagebreak
  \textit{E-mail}: \texttt{dkral@fi.muni.cz}
  
  \medskip
  
\noindent  A.~Lamaison, \textsc{Faculty of Informatics, Masaryk University, Botanick\'a 68A, 602 00 Brno, Czech Republic} \\
  \indent\qquad\textsc{Previous affiliation: Mathematics  Department, Freie Universit\"at, Arnimallee 3, 14195 Berlin, Germany}\par\nopagebreak
  \textit{E-mail}: \texttt{lamaison@fi.muni.cz}
  
  \medskip
  
\noindent  L.~Vena, \textsc{Department of Applied Mathematics (KAM), Charles University, Malo\-stransk\'e n\'am\v{e}st\'i 25, Praha 1, Czech Republic} \\
  \indent\qquad\textsc{Department of Mathematics, Universitat Polit\`ecnica de Catalunya, Barce\-lona, Spain}
\par\nopagebreak
  \textit{E-mail}: \texttt{lluis.vena@gmail.com}

}}

\maketitle
\begin{abstract}
\noindent The fractional and circular chromatic numbers are the two most studied non-integral refinements of the chromatic number of a graph.
Starting from the definition of a coloring base of a graph, which originated in work related to ergodic theory, 
we formalize the notion of a gyrocoloring of a graph: the vertices are colored by translates of a single Borel set in the circle group, and neighbouring vertices receive disjoint translates.
The corresponding gyrochromatic number of a graph always lies between the fractional chromatic number and the circular chromatic number.
We investigate basic properties of gyrocolorings. In particular, we construct examples of graphs
whose gyrochromatic number is strictly between the fractional chromatic number and the circular chromatic number. We also establish several equivalent definitions of the gyrochromatic number, including a version involving all finite abelian groups.
\end{abstract}

\section{Introduction}
\label{sec:intro}

Graph coloring is one of the most studied topics in graph theory.
In order to refine the basic notion of the chromatic number of a graph,
various non-integral relaxations were introduced,
in particular, to capture how close a graph is to being colorable with fewer colors.
Among them,
the two most intensively studied notions are the \emph{circular chromatic number} and
the \emph{fractional chromatic number}.
We build on the work of Avila and the first author~\cite{AviC16}
who introduced a notion of a coloring base of a graph
in relation to applications of their new proof of a generalization of Rokhlin's lemma;
this notion leads to a chromatic parameter of a graph
which lies between the circular and fractional chromatic numbers.
The purpose of the present article is to introduce this parameter,
which we refer to as the \emph{gyrochromatic number} of a graph,
in the context of graph coloring.

We begin by recalling the notions of circular and fractional colorings and fixing some notation.
All graphs in this paper are finite and simple.
If $G$ is a graph, then $V(G)$ and $E(G)$ are its vertex and edge sets, and $|G|$ is the number of its vertices.
The \emph{chromatic number} $\chi(G)$ of a graph $G$ is the smallest integer $k$ for which
there exists a mapping $f:V(G)\to [k]$ such that $f(u) \not= f(v)$ for every edge $uv$ of $G$;
we use $[k]$ to denote the set of the first $k$ positive integers.
The \emph{circular chromatic number} $\chi_c(G)$ of a graph $G$
is the smallest real $z$ for which
there exists a mapping $f:V(G)\to [0,z)$ such that $1\le |f(v)-f(u)|\le z-1$ for every edge $uv$
(it can be shown that the minimum is always attained).
The mapping $f$ can be viewed as a mapping of the vertices of $G$ to unit-length arcs of a circle of circumference $z$ such that
adjacent vertices are mapped to internally disjoint arcs.
We remark that there are several equivalent definitions of the circular chromatic number~\cite{Zhu01},
e.g., through homomorphisms to particular graphs or through balancing edge-orientations.
In relation to the chromatic number, it is not hard to show that every graph $G$ satisfies the following:
\[\chi(G)-1<\chi_c(G)\le\chi(G),\]
in particular $\chi(G)=\lceil\chi_c(G)\rceil$.
The circular chromatic number was introduced by Vince~\cite{Vin88} in the late 1980s and
has been the main subject of many papers since then;
we refer to the survey by Zhu~\cite{Zhu01} for a detailed exposition.

Coloring vertices of a graph $G$ can be viewed as an integer program such that
independent sets in $G$ are assigned weights zero and one in such a way that
every vertex belongs to independent sets whose total weight is (at least) one and
the sum of the weights of all independent sets is minimized.
Fractional coloring is a linear relaxation of this optimization problem:
the \emph{fractional chromatic number} $\chi_f(G)$ is the smallest real $z$ for which
there is an assignment of non-negative weights to independent sets in $G$
such that the sum of their weights is $z$ and
each vertex belongs to independent sets whose total weight is at least one
(it can be shown that the minimum is always attained).
Equivalently,
$\chi_f(G)$ can be defined as the smallest real $z$ such that
each vertex $v$ of $G$ can be assigned a Borel subset of $[0,z)$ of measure one
in such a way that adjacent vertices are assigned disjoint subsets.
It can be shown that the circular chromatic number lies between the fractional chromatic number and chromatic number for every graph $G$:
\[\chi_f(G)\le\chi_c(G)\le\chi(G).\]
Both inequalities can be strict, and
the gap between the fractional chromatic number and chromatic number (and
so the circular chromatic number) can be arbitrarily large.
For instance, the chromatic number of Kneser graphs can be arbitrarily large and
their fractional chromatic number arbitrarily close to two.
We recall that the Kneser graph $K(n,k)$ has $\binom{n}{k}$ vertices that
are viewed as corresponding to $k$-element subsets of $[n]$;
two vertices are adjacent if the corresponding subsets are disjoint.
The chromatic number of $K(n,k)$ is $n-2k+2$ (if $n\ge 2k$)
by the famous result of Lov\'asz~\cite{Lov78}, and
their fractional chromatic number is known to be equal to $n/k$.
We refer to the book by Scheinerman and Ullman~\cite{SchU11}
for further results on fractional coloring and
fractional graph parameters in general.

We next recall the notion of a coloring base of
a graph introduced in~\cite[Definition 3.8]{AviC16},
which is the starting point for the present discussion.
Let $G$ be a graph and $Z$ be an abelian group.
A set $A\subseteq Z$ is a \emph{coloring $Z$-base} for $G$
if there exists a function $f:V(G)\to A$ such that the sets $A+f(u)$ and $A+f(v)$ are disjoint for every edge $uv$ of $G$;
we write just \emph{coloring base} if the group $Z$ is clear from the context.
For a topological group $Z$ equipped with a Haar probability measure $\mu$,
we define
\begin{equation}
\label{eq:sigma}
\sigma_Z(G)=\sup\{\mu(A):\mbox{ $A$ is a Borel coloring $Z$-base for $G$}\}.
\end{equation}
This notion is related to results in ergodic theory, and
we refer the reader to~\cite{AviC16} for the exposition of this relation.
We will be particularly interested in the group $\TT=\RR/\ZZ$ and
the groups $\ZZ_N=\ZZ/N\ZZ$ for $N\in\NN$.
The former can be viewed as given by the unit interval $[0,1]$ with $0$ and $1$ identified and the usual Borel measure, and
the latter is simply $\{0,\ldots,N-1\}$ with addition modulo $N$ equipped with the uniform discrete probability measure.

The notion of a coloring base resembles equivalent definitions of circular and fractional chromatic number,
which can be cast using the group $\TT$ as follows.
The circular chromatic number of a graph $G$ is the inverse of the maximum $\mu(A)$
where $A$ is a \emph{connected} coloring $\TT$-base for $G$ (here, connected means as a subset of $\TT$).
The fractional chromatic number of a graph $G$ is the inverse of the maximum $z$ such that
each vertex of $G$ can be assigned a Borel subset of $\TT$ with measure $z$ and
adjacent vertices are assigned disjoint subsets.
Informally speaking,
circular coloring assigns arcs to vertices of $G$ and
fractional coloring assigns Borel subsets to vertices of $G$.
A coloring $\TT$-base is tighter than fractional coloring since vertices must be assigned rotational copies of the same Borel subset, and it is 
looser than circular coloring since the Borel subset assigned to vertices need not be an arc.
This leads us to the following definition:
the \emph{gyrochromatic number} $\chi_g(G)$ of a graph $G$ is the inverse of $\sigma_{\TT}(G)$.
The equivalent definitions of the circular and fractional chromatic numbers given above
yield that the gyrochromatic number of every graph $G$ lies 
between its fractional and circular chromatic numbers:
\begin{equation}
\label{eq:basicineqs}
\chi_f(G)\leq \chi_g(G)\leq \chi_c(G).
\end{equation}
Similarly to the notions of the fractional and circular chromatic number,
the definition of the gyrochromatic number of a graph is robust in the sense that
it can be cast in several different ways.
In Section~\ref{sec:eq}, we show that the definition through coloring $\TT$-bases
yields the same notion as its discrete variant using $\ZZ_N$-bases (cf.~Corollary~\ref{cor:eqN}).
Coloring, circular coloring and fractional coloring of a graph can be defined by using homomorphisms to special classes of graphs:
cliques, circular cliques and Kneser graphs.
This holds also for the gyrochromatic number of a graph,
which can be defined in terms of homomorphisms to circulant graphs (cf.~Theorem~\ref{thm:eqC}).
More interestingly,
in Section~\ref{sec:torus} we prove that
the definition stays the same
when considering higher-dimensional-torus analogues of $\TT$ (cf.~Theorem~\ref{thm:eqd}),
which implies that the discrete variant of the definition
is the same when all finite abelian groups are considered
instead of the groups $\ZZ_N$ only (cf.~Corollaries~\ref{cor:eqZ} and~\ref{cor:eqA}).
We note that the fractional chromatic number can be cast in a similar way 
but taking into account all finite (not necessarily abelian) groups (see the discussion after Corollary~\ref{cor:eqA}),
i.e.,
the gyrochromatic number can be viewed as an analogue of the fractional chromatic number restricted to abelian groups.
We believe that these properties show that the gyrochromatic number of a graph is a natural and robust parameter,
which is likely to play an important role in providing a more detailed understanding of the structure of graphs
whose circular and fractional chromatic numbers differ. 

In addition to presenting several equivalent definitions of the gyrochromatic number in Sections~\ref{sec:eq} and~\ref{sec:torus},
we also establish some of its basic properties and
in particular construct graphs with gyrochromatic number strictly
between the circular and fractional chromatic numbers (cf.~Theorem~\ref{thm:sandwich}).
Finally, in Section~\ref{sec:attain}
we show that, somewhat surprisingly, the supremum in the definition \eqref{eq:sigma} need not be attained,
which also implies that the infimum in its discrete variant (as given in Corollary~\ref{cor:eqN})
need not be attained.

\section{Equivalent definitions}
\label{sec:eq}

In this section,
we present alternative definitions of the gyrochromatic number, and
show that these are equivalent to the original definition stated in Section~\ref{sec:intro}. We begin by giving another definition,
which is analogous to circular coloring and provides the notion of a gyrocoloring.
If $z$ is a non-negative real, a \emph{$z$-gyrocoloring} of a graph $G$
is a mapping $g$ from $V(G)$ to Borel subsets of $[0,z)$ such that
the measure of each of $g(u)$, $u\in V(G)$, is one,
the sets $g(u)$ and $g(v)$ are disjoint for every edge $uv$, and
the sets $g(u)$ and $g(v)$ are rotationally equivalent for any two vertices $u$ and $v$ of $G$,
i.e., there exists $x\in [0,z)$ such that
\[g(u)=(g(v)+x)\mod z=\{(y+x)\mod z: y\in g(v)\}.\]
The gyrochromatic number of $G$ is the infimum of all $z$ such that $G$ has a $z$-gyrocoloring.
The equivalence of this definition to the one given in Section~\ref{sec:intro} is rather easy to see.
For completeness, we include a short proof.

\begin{proposition}
\label{prop:color}
Let $G$ be a graph.
For every positive real $z$,
the graph $G$ has a $z$-gyrocoloring if and only if
$G$ has a coloring $\TT$-base of measure $z^{-1}$.
\end{proposition}

\begin{proof}
Fix a graph $G$ and a positive real $z$.
Suppose that $g$ is a $z$-gyrocoloring.
Fix a vertex $v_0$ of $G$ and let $x_v\in [0,z)$ be such that $g(v)=(g(v_0)+x_v)\mod z$;
in particular, $x_{v_0}=0$.
The set
\[\{x/z:x\in g(v_0)\}\]
is a coloring $\TT$-base;
this can be seen by setting the function $f$ from the definition of a coloring base to be $f(v)=g(v)/z$.

For the other direction, let $A$ be a coloring $\TT$-base of measure $z^{-1}$, and
let $f$ be the function from the definition of a coloring base.
We can now define a $z$-gyrocoloring of the graph $G$ as follows:
\[g(v)=\{(x+f(v))\cdot z \mod z: x\in A\}\]
for every vertex $v\in V(G)$.
\end{proof}

We next turn our attention to definitions where the equivalence is more complicated to see, and
start with showing that the discrete and Borel variants of the definition of the gyrochromatic number
are equivalent.

\begin{theorem}
\label{thm:eqN}
Let $G$ be a graph.
It holds that $\sigma_{\ZZ_N}(G)\le\sigma_{\TT}(G)$ for every $N\in\NN$ and
\[\sigma_{\TT}(G)=\sup_{N\in\NN}\sigma_{\ZZ_N}(G)=\lim_{N\to\infty}\sigma_{\ZZ_N}(G).\]
\end{theorem}

\begin{proof}
Fix a graph $G$.
If $A$ is a coloring $\ZZ_N$-base for $G$,
then the set
\[\bigcup_{x\in A}\left[\frac{x-1}{N},\frac{x}{N}\right)\]
is a coloring $\TT$-base.
Hence, $\sigma_{\TT}(G)\ge\sigma_{\ZZ_N}(G)$ for every $N\in\NN$.
We therefore have
\[\sigma_{\TT}(G)\ge\sup_{N\in\NN}\sigma_{\ZZ_N}(G)\ge \limsup_{N\to\infty}\sigma_{\ZZ_N}(G).\]
We now prove that 
\[\lim_{N\to\infty}\sigma_{\ZZ_N}(G)=\sigma_{\TT}(G),\]
which will complete the proof.

Choose any $\varepsilon>0$ and let $A$ be a coloring $\TT$-base such that $\mu(A)\geq \sigma_{\TT}(G)-\varepsilon$.
We may assume without loss of generality that $\mu(A)>0$.
Let $f:V(G)\to\TT$ be a mapping such that $A+f(u)$ and $A+f(v)$ are disjoint for every edge $uv$ of the graph $G$.

Since the measure $\mu$ is regular, there exists a closed set $B\subseteq A$ such that $\mu(A\setminus B)<\varepsilon$.
For $m\in\NN$, let $B_m:=B +\big[ \frac{-1}{m},\frac{1}{m} \big]$;
since the set $B$ is closed, it holds that
\[B=\bigcap_{m\in\NN} B_m\mbox{ and }\lim_{m\to\infty}\mu(B_m)=\mu(B).\]
Hence, there exists an integer $m$ such that $\mu(B_m\setminus B)< \varepsilon$,
whence $\mu(B_m\triangle A)\leq \mu(B_m\triangle B)+\mu(B\triangle A)<2\varepsilon$;
fix such $m$ for the rest of the proof.
Note that the measure of $B_m$ is at least $\mu(A)-2\varepsilon\ge \sigma_{\TT}(G)-3\varepsilon$.

Choose $N\in\NN$ such that $m/N\le\varepsilon$ and define $f'(v)=\lfloor Nf(v)\rfloor/N$;
here and in what follows, multiplications such as $Nf(v)$ mean that $f(v)$ is viewed as an element of $[0,1]$ and
is multiplied in $\RR$ by $N$.
Since the sets $A+f(u)$ and $B_m+f(u)$ differ on a set of measure at most $2\varepsilon$ for every vertex $u\in V(G)$,
the measure of the intersection $B_m+f(u)$ and $B_m+f(v)$ is at most $4\varepsilon$ for every edge $uv$ of $G$.
Since the set $B_m$ has at most $m/2$ connected components,
i.e., $B_m$ viewed as a subset of a circle consists of at most $m/2$ closed arcs,
the measure of the intersection $B_m+f'(u)$ and $B_m+f'(v)$ is at most $4\varepsilon+m/N$.
Choose $z\in\TT$ randomly according to $\mu$ and define a set $A(z)\subseteq\ZZ_N$ to be the set containing all $i\in\{0,\ldots,N-1\}$ such that
$z+i/N\in B_m$ and \[z+i/N\not\in B_m+f'(v)-f'(u)\] for every edge $uv$ of $G$.
The probability that a particular $i$ satisfies that $z+i/N\in B_m$ is equal to the measure of $B_m$. The probability that $z+i/N$ is both in $B_m$ and in $B_m+f'(v)-f'(u)$ for a particular edge $uv$
is equal to the measure of the intersection of $B_m+f'(u)$ and $B_m+f'(v)$, which is at most $4\varepsilon+m/N$.
Hence, the probability that $i$ is included in the set $A(z)$ is at least
\[\mu(A)-2\varepsilon-(4\varepsilon+m/N)|E(G)|\ge \mu(A)-2\varepsilon-5\varepsilon|E(G)|.\]
The expected size of $A(z)$ is therefore at least
\[N(\mu(A)-2\varepsilon-5\varepsilon|E(G)|).\]
Fix $z$ such that the size of $A(z)$ is at least the expected size.

We next show that $A(z)$ is a coloring $\ZZ_N$-base for $G$.
Consider the function $f'':V(G)\to\ZZ_N$ defined as $f''(v)=Nf'(v)$ and
observe that the sets $A(z)+f''(u)$ and $A(z)+f''(v)$ are disjoint for every edge $uv$.
Indeed, if the intersection of $A(z)+f''(u)$ and $A(z)+f''(v)$ were non-empty,
there would exist $i,j \in A(z)$ such that $i+f''(u)=j+f''(v)$,
which would imply that $z+i/N\in B_m$ and $z+j/N\in B_m$;
since $z+i/N=z+j/N+f'(v)-f'(u)$, it would then follow that $z+i/N\in B_m+f'(v)-f'(u)$,
contradicting the definition of $A(z)$.
Hence $A(z)$ is a coloring $\ZZ_N$-base for $G$, so $\sigma_{\ZZ_N}(G)\ge\frac{|A(z)|}{N}=\mu(A)-2\varepsilon-5\varepsilon|E(G)|$.

We have thus proved that for all $N$ sufficiently large (depending on $\varepsilon$) we have
\[
\sigma_{\TT}(G)\geq \sigma_{\ZZ_N}(G)\ge \sigma_{\TT}(G)-3\varepsilon-5\varepsilon|E(G)|.\]
Since the choice of $\varepsilon>0$ was arbitrary, we deduce that
\[\sigma_{\TT}(G)=\lim_{N\to\infty} \sigma_{\ZZ_N}(G)\]
and the result follows.
\end{proof}

Theorem~\ref{thm:eqN} yields the following.

\begin{corollary}
\label{cor:eqN}
The gyrochromatic number of a graph $G$ is equal to the infimum of $N/K$ for which
there exists a $K$-element set $A\subseteq\ZZ_N$ and a function $f:V(G)\to\ZZ_N$
such that $A+f(u)$ and $A+f(v)$ are disjoint for every edge $uv$ of $G$.
\end{corollary}

We next turn our attention to a definition through homomorphisms to circulant graphs.
Fix an integer $N$ and a set $S\subseteq\ZZ_N$ such that $S=-S$ and $0\not\in S$;
the \emph{circulant graph} $C(N,S)$ is the graph with vertex set $\ZZ_N$ such that
two vertices $i$ and $j$ of $C(N,S)$ are adjacent if $j-i\in S$.
Observe that the fractional chromatic number and the gyrochromatic number of every circulant graph is the same.
To see this, fix a circulant graph $C(N,S)$.
Since the graph $C(N,S)$ is vertex-transitive,
it holds that $\chi_f(C(N,S))=N/\alpha(C(N,S))$
where $\alpha(C(N,S))$ is the independence number of $C(N,S)$.
On the other hand,
every independent set in the graph $C(N,S)$ is a coloring $\ZZ_N$-base of $C(N,S)$
with the function $f$ defined as $f(x)=x$ for $x\in\ZZ_N$,
which implies that $\chi_g(C(N,S))\le N/\alpha(C(N,S))$ by Theorem~\ref{thm:eqN}.
Hence, it holds that $\chi_f(C(N,S))=\chi_g(C(N,S))=N/\alpha(C(N,S))$.

If $G$ and $H$ are two graphs, we say that $G$ is \emph{homomorphic} to $H$
if there exists a mapping $h:V(G)\to V(H)$ such that
$h(u)h(v)$ is an edge of $H$ for every edge $uv$ of $G$.
The mapping $h$ is a \emph{homomorphism} from $G$ to $H$.
If $G$ is homomorphic to $H$,
we also write $G\to H$ and say that $H$ \emph{admits a homomorphism} from $G$.

\begin{theorem}
\label{thm:eqC}
The gyrochromatic number of a graph $G$ is equal to
\[\inf_{\substack{N\in\NN, S\subseteq\ZZ_N \\ G\to C(N,S)}}\chi_f\left(C(N,S)\right)=
  \inf_{\substack{N\in\NN, S\subseteq\ZZ_N \\ G\to C(N,S)}}\frac{N}{\alpha\left(C(N,S)\right)}\;,\]
i.e., is equal to the infimum of the fractional chromatic numbers of the circulant graphs that admit a homomorphism from $G$.
\end{theorem}
\begin{proof}
We will show that the following holds for every $N\in\NN$:
\begin{equation}
\sigma_{\ZZ_N}(G)=\max_{\substack{S\subseteq\ZZ_N \\ G\to C(N,S)}}\frac{\alpha\left(C(N,S)\right)}{N}.\label{eq:eqC}
\end{equation}
Since the graph $C(N,S)$ is vertex transitive for every choice of $S$,
it holds that the fractional chromatic number of $C(N,S)$ is equal to $N/\alpha\left(C(N,S)\right)$.
In particular, the statement of the theorem will follow from \eqref{eq:eqC} and Theorem \ref{thm:eqN}.

Fix $N\in\NN$. We prove the equality in \eqref{eq:eqC} as two inequalities,
starting with showing that
\[\sigma_{\ZZ_N}(G)\le\max_{\substack{S\subseteq\ZZ_N \\ G\to C(N,S)}}\frac{\alpha\left(C(N,S)\right)}{N}.\]
Let $A$ be a coloring $\ZZ_N$-base for $G$ and let $f:V(G)\to\ZZ_N$ be such that
$A+f(u)$ and $A+f(v)$ are disjoint for every edge $uv$ of $G$.
We define $S$ as the set containing all elements $x$ such that
there is an edge, say $uv$, of $G$ and $x$ is equal to $f(u)-f(v)$ or $f(v)-f(u)$.
We claim that the set $A$ is independent in the circulant graph $C(N,S)$:
if $A$ was not independent,
then there would exist $i,j\in A$ such that $i=j+f(v)-f(u)$ for an edge $uv$,
which would imply that the sets $A+f(u)$ and $A+f(v)$ are not disjoint ($i+f(u)$ would be their common element).
Hence, the independence number of $C(N,S)$ is at least $|A|$.
Since it holds that the vertices $f(u)$ and $f(v)$ of $C(N,S)$ are adjacent for every edge $uv$ of $G$,
the mapping $f$ is a homomorphism from $G$ to $C(N,S)$ and the inequality follows.

We next prove that
\[\sigma_{\ZZ_N}(G)\ge\max_{\substack{S\subseteq\ZZ_N \\ G\to C(N,S)}}\frac{\alpha\left(C(N,S)\right)}{N}.\]
Fix a set $S\subseteq\ZZ_N$ such that
there exists a homomorphism $f:V(G)\to\ZZ_N$ from $G$ to the circulant graph $C(N,S)$, and
let $A$ be an independent set of $C(N,S)$ of size $\alpha(C(N,S))$.
We claim that $A$ is a coloring $\ZZ_N$-base for $G$.
To establish this, it is enough to show that $A+f(u)$ and $A+f(v)$ are disjoint for every edge $uv$ of $G$.
Since $A$ is an independent set in $G$, the sets $A$ and $A+x$ are disjoint for every $x\in S$.
Consider an edge $uv$ of $G$.
Since $f$ is a homomorphism from $G$ to $C(N,S)$, it follows that $f(v)-f(u)\in S$,
in particular, the sets $A$ and $A+f(v)-f(u)$ are disjoint.
Hence, the sets $A+f(u)$ and $A+f(v)$ are disjoint.
We conclude that $A$ is a coloring $\ZZ_N$-base.
\end{proof}

We remark that in Section~\ref{sec:torus} we establish a more general statement than that of Theorem~\ref{thm:eqC};
circulant graphs are Cayley graphs of the (abelian) group $\ZZ_N$ and
we will show that the gyrochromatic number of a graph $G$ is equal
to the infimum of the fractional chromatic numbers of Cayley graphs of finite abelian groups that
admit a homomorphism from $G$ (cf.~Corollary~\ref{cor:eqA}).
We remark that the fractional chromatic number of $G$ is equal
to the minimum of the fractional chromatic numbers of Cayley graphs of finite groups that
admit a homomorphism from $G$ (see the discussion after Corollary~\ref{cor:eqA}).

\section{Simple properties}
\label{sec:basic}

In this section, we establish several simple properties of the gyrochromatic number;
some of them will be used within our arguments later in the paper.
We start with two properties related to products of graphs.
Recall that the \emph{Cartesian product} $G\Box H$ of two graphs $G$ and $H$
is the graph with vertex set $V(G)\times V(H)$ such that two vertices $(v,v')$ and $(w,w')$ of $G\Box H$
are adjacent if either $v=w$ and $v'w'$ is an edge of $H$ or $v'=w'$ and $vw$ is an edge of $G$.
Similarly to the chromatic number and the circular chromatic number,
the gyrochromatic number is also preserved by the Cartesian product of a graph with itself.

\begin{proposition}
\label{prop:prod}
For every graph $G$, it holds that $\chi_g(G)=\chi_g(G\Box G)$.
\end{proposition}

\begin{proof}
We show that $A\subseteq\TT$ is a coloring base for $G$ if and only if it is a coloring base for $G\Box G$.
If $A$ is a coloring base for $G\Box G$, then it is clearly a coloring base for $G$.
Hence, we focus on proving the reverse implication.
Let $A$ be a coloring $\TT$-base for $G$ and $f:V(G)\to\TT$ such that $A+f(v)$ and $A+f(w)$ are disjoint for any edge $vw$ of $G$.
We define $g:V(G)^2\to\TT$ by setting $g(v,v')$ to be $f(v)+f(v')$, and
show that if $(v,v')$ and $(w,w')$ are adjacent in $G\Box G$,
then $A+g(v,v')$ and $A+g(w,w')$ are disjoint.

Suppose that $(v,v')$ and $(w,w')$ are adjacent. It holds that either $v=w$ or $v'=w'$.
In the former case, $v'$ and $w'$ are adjacent in $G$ and so $A+f(v')$ and $A+f(w')$ are disjoint.
Consequently, $A+g(v,v')=A+f(v)+f(v')$ and $A+g(w,w')=A+f(w)+f(w')=A+f(v)+f(w')$ are disjoint.
Since the case $v'=w'$ is symmetric,
$A$ is a coloring base for the graph $G\Box G$.
\end{proof}

Proposition~\ref{prop:prod} does not generalize to Cartesian products of distinct graphs,
however the following holds.

\begin{proposition}
\label{prop:prod2}
For any two graphs $G$ and $H$, it holds that $\chi_g(G\Box H)=\chi_g(G\cup H)$,
where $G\cup H$ is the disjoint union of $G$ and $H$.
\end{proposition}

\begin{proof}
Since $G\Box H$ is a component of the graph $(G\cup H)\Box(G\cup H)$,
it holds that $\chi_g(G\Box H)\le\chi_g(G\cup H)$ by Proposition~\ref{prop:prod}.
On the other hand, both $G$ and $H$ are subgraphs of $G\Box H$,
which implies that if $A\subseteq\TT$ is a coloring base for $G\Box H$,
then it is also a coloring base of $G\cup H$.
Indeed,
if $f:V(G\Box H)\to\TT$ is a function such that $A+f(v)$ and $A+f(w)$ are disjoint for any edge $vw$ of $G\Box H$,
then its restriction to the copies of $G$ and $H$ in $G\Box H$ witness that $A$ is a coloring base for $G\cup H$.
Hence, $\chi_g(G\cup H)\le\chi_g(G\Box H)$.
\end{proof}

We derive the following from Proposition~\ref{prop:prod2};
note that a graph $G+H$ in the statement is not uniquely determined by the graphs $G$ and $H$.

\begin{proposition}
\label{prop:identify}
Let $G$ and $H$ be any two graphs and
let $G+H$ be a graph obtained by identifying a vertex of $G$ and a vertex of $H$.
It holds that $\chi_g(G\cup H)=\chi_g(G+H)$.
\end{proposition}

\begin{proof}
Since both $G$ and $H$ are subgraphs of $G+H$, it holds that $\chi_g(G\cup H)\le \chi_g(G+H)$.
On the other hand, $G+H$ is a subgraph of $G\Box H$ and
so it holds that $\chi_g(G+H)\le\chi_g(G\Box H)=\chi_g(G\cup H)$,
where the last equality follows from Proposition~\ref{prop:prod2}.
\end{proof}

In the next two paragraphs, we give an example of two (small) graphs $G$ and $H$ such that
$\chi_g(G\Box H)>\max\{\chi_g(G),\chi_g(H)\}$,
which is equivalent to $\chi_g(G\cup H)>\max\{\chi_g(G),\chi_g(H)\}$ by Proposition~\ref{prop:prod2}.
The graph $G\cup H$ obtained as the disjoint union of the graphs $G$ and $H$ also satisfies that
$\chi_f(G\cup H) < \chi_g(G\cup H) < \chi_c(G\cup H)$.
Hence, a graph $G+H$ obtained by identifying a vertex of $G$ with a vertex of $H$
is an example of a connected graph satisfying that $\chi_f(G+H)<\chi_g(G+H)<\chi_c(G+H)$
by Proposition~\ref{prop:identify}.

Recall that the \emph{lexicographic product} $G[G']$ of graphs $G$ and $G'$
is the graph whose vertex set is $V(G)\times V(G')$ and two vertices $(v,v')$ and $(w,w')$
are adjacent if either $vw$ is an edge of $G$ or $v=w$ and $v'w'$ is an edge of $G'$.
Consider graphs $G=K_5$ and $H=K_2[C_5]$
and note that they both are circulant graphs
(the graph $K_2[C_5]$ is isomorphic to the circulant graph $C(10,S)$ for $S=\left\{1,3,4,5,6,7,9\right\}$),
and therefore their gyrochromatic number
is equal to their fractional chromatic number (see the discussion before Theorem~\ref{thm:eqC}),
which is five for each of them.
Clearly $\chi(G) = \chi_c(G) = 5$.
Moreover, since the complement of $H$ is  the union of two disjoint copies of $C_5$,
we can see that $\chi(H) = 6$, and also (since this complement is therefore disconnected) that
$\chi_c(H) = \chi(H)$ by~\cite[Corollary~3.1]{Zhu01}.
We next observe that $\alpha(G \Box H) \leq 9$:
since $G$ is a clique of size five, the independence number of $G \Box H$ is
the maximum total number of vertices assembled from five
pairwise disjoint independent sets of $H$.
As $\alpha(H)=2$ and $H$ is not $5$-colorable,
the maximum number of vertices of $H$ contained in five (disjoint) independent sets of $H$ is at most $9$.
We conclude that $\alpha(G \Box H) \leq 9$, which implies that $\chi_f(G\Box H) \geq 50/9$.
Since $G\Box H$ is a Cayley graph for the group $\ZZ_2\times\ZZ_5^2$,
it holds that $\chi_f(G\Box H)=\chi_g(G\Box H)=50/9$ by Proposition~\ref{prop:prod2} and Corollary~\ref{cor:eqZ}.
However, since we prove Corollary~\ref{cor:eqZ} only in Section~\ref{sec:torus},
we next provide an argument not depending on Corollary~\ref{cor:eqZ}.

Since $\chi_g(G\Box H) \geq 50/9$,
Proposition~\ref{prop:prod2} yields that
$\chi_g(G\cup H)\geq 50/9>\max\{\chi_g(G),\chi_g(H)\}=\max\{\chi_f(K_5),\chi_f(H)\}=5$.
We next show that $\chi_g(G\cup H)<6$ by establishing that $G\cup H$ admits a $40/7$-gyrocoloring.
The structure of such a coloring is outlined in Figure~\ref{fig:gyrocoloring}.
Namely, consider the set $[0,\frac{1}{2}) \cup [\frac{15}{14},\frac{22}{14})$
and assign its five copies shifted by
$0$, $\frac{1}{2}$, $\frac{29}{14}$, $\frac{36}{14}$ and $\frac{58}{14}$ to the vertices of $G$,
its five copies shifted by $\frac{8i}{7}$ for $i = 0, 1, 2, 3, 4$ to the vertices of one copy of $C_5$ in $H$ and
its five copies shifted by $\frac{8i+4}{7}$ for $i = 0, \dots, 4$ to the vertices of the other copy of $C_5$ in $H$.
Furthermore, recalling the fractional and circular chromatic numbers of 
$G$ and $H$, we get that $\chi_f(G\cup H) = 5$ and $\chi_c(G\cup H) = 6$
which yields that
$\chi_f(G\cup H) < \chi_g(G\cup H) < \chi_c(G\cup H)$.

\begin{figure}
\begin{center}
\includegraphics[scale=0.6]{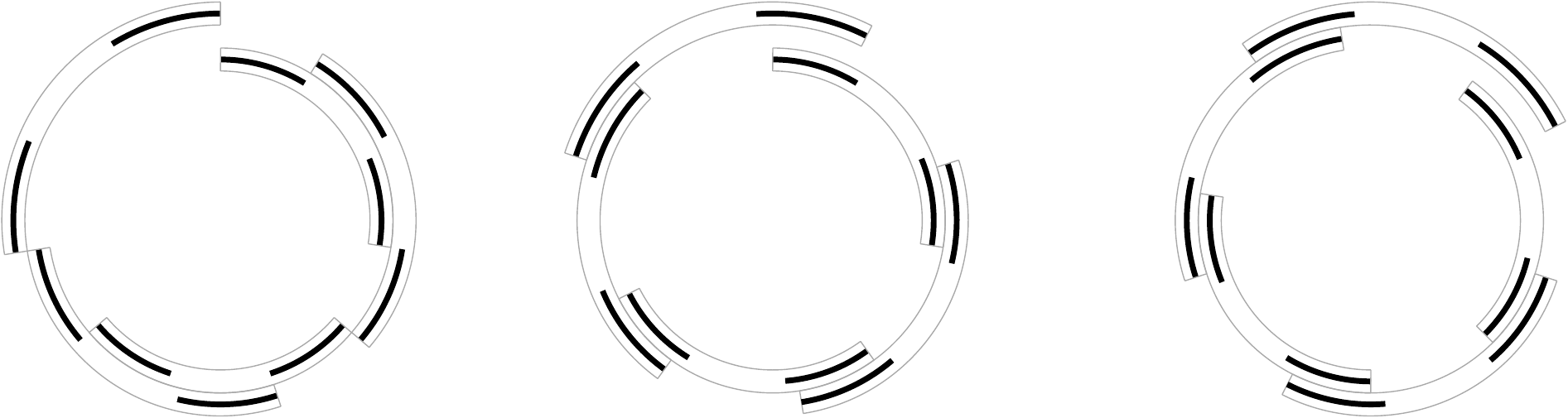}
\end{center}
\caption{Outline of the gyrocoloring of $K_5 \cup K_2[C_5]$.
The five sets depicted on the left-hand side correspond to the vertices of $K_5$
(the sets are pairwise disjoint).
The middle correspond to the vertices of induced $C_5$,
and similarly the right-hand side correspond to the other $C_5$
(the union of the sets in the middle is disjoint from the union of the sets on the right).}
\label{fig:gyrocoloring}
\end{figure}

Considering lexicographic products, Gao and Zhu~\cite{GaoZ96} showed that 
if $\chi_f(G)=\chi_c(G)$, then $\chi_c(G[H])=\chi_c(G)\;\chi(H)$.
In particular, $\chi_c(K_k[H]) = k\;\chi(H)$.
Recall that if $H$ is a circulant graph,
then $K_k[H]$ is also a circulant graph and thus
we get that $\chi_f(K_k[H]) = \chi_g(K_k[H]) = k\;\chi_f(H)$.
Hence, if a circulant graph $H$ satisfies $\chi_f(H)<\chi(H)$,
then $K_k[H]$ is an example of a graph such that the gyrochromatic and circular chromatic numbers differ;
the graph $K_2[C_5]$ discussed above is a particular case of this more general argument.

We conclude this section with a lemma,
which can be used to prove a lower bound on the gyrochromatic number of a graph that
is larger than the fractional chromatic number.
The lemma in particular yields an example of a graph such that
its gyrochromatic number is strictly larger than its fractional chromatic number:
consider the line graph $G$ of the Petersen graph and note that $\omega(G)=\chi_f(G)=3$ and $\chi(G)=4$.

\begin{lemma}
\label{lm:fractclique}
Let $G$ be an $n$-vertex graph. If $\omega(G)<\chi(G)$, then $\chi_g(G)\ge\frac{n}{n-1}\omega(G)$.
\end{lemma}

\begin{proof}
Let $k=\omega(H)$ and let $v_1,\ldots,v_k$ be vertices of a clique with $k$ vertices in $G$.
Suppose that $A$ is a coloring $\TT$-base for $G$ with measure $\delta>0$.
Finally, let $f:V(G)\to\TT$ be such that $A+f(u)$ and $A+f(v)$ are disjoint for every edge $uv$ of $G$.
In particular, the sets $A+f(v_1),\ldots,A+f(v_k)$ are pairwise disjoint.

For $x\in\TT$ and $i\in [k]$, define $V_i(x)\subseteq V(G)$ as the set of vertices $v$ such that $x+f(v_i)\in A+f(v)$.
Since the element $x+f(v_i)$ is contained in the set $A+f(v)$ for every $v\in V_i(x)$,
each of the sets $V_i(x)$ is an independent set in $G$.
In addition, the sets $V_1(x),\ldots,V_k(x)$ are pairwise disjoint for every $x\in\TT$.
Indeed, if a vertex $v$ were contained in $V_i(x)$ and $V_j(x)$, $i\not=j$, $i,j\in [k]$,
then both $x+f(v_i)$ and $x+f(v_j)$ would be contained in $A+f(v)$;
this would imply that the element $x-f(v)+f(v_i)+f(v_j)$ is contained both in $A+f(v_i)$ and $A+f(v_j)$,
which contradicts that the sets $A+f(v_i)$ and $A+f(v_j)$ are disjoint.

Since the graph $G$ is not $k$-colorable,
the union $V_1(x)\cup\cdots\cup V_k(x)$ does not contain all vertices of $G$ for any $x\in\TT$.
Hence, for every $x\in\TT$ we have 
\[\left|V_1(x)\cup\cdots\cup V_k(x)\right|=\sum_{i=1}^k\left|V_i(x)\right|\le n-1,\]
which implies that $\int_{\TT}\sum_{i=1}^k\left|V_i(x)\right|\dd x\le n-1$. On the other hand it holds that
\[\int_{\TT}\sum_{i=1}^k\left|V_i(x)\right|\dd x=\sum_{v\in V(G)}\sum_{i=1}^k\mu\left(A+f(v)-f(v_i)\right)=kn\delta.\]
It follows that $\delta$ is at most $\frac{n-1}{kn}$.
We conclude that $\sigma_{\TT}(G)\le\frac{n-1}{kn}$,
which yields that the gyrochromatic number of $G$ is at least $\frac{n}{n-1}k$.
\end{proof}

\section{Universality}
\label{sec:torus}

In this section,
we show that a higher-dimensional-torus analogue of the gyrochromatic number
is equal to the (one-dimensional) gyrochromatic number and
use this result to prove a generalization of Theorem~\ref{thm:eqC} to all finite abelian groups.
As the first step towards this result,
we observe that the proof of Theorem~\ref{thm:eqN} readily translates to higher dimensions;
we formulate the corresponding statement as a lemma.

\begin{lemma}
\label{lm:eqd}
Let $d$ be any positive integer. For every graph $G$, it holds that
\[\sigma_{\TT^d}(G)=\sup_{N\in\NN}\sigma_{\ZZ_N^d}(G)=\lim_{N\to\infty}\sigma_{\ZZ_N^d}(G).\]
\end{lemma}

We next state and prove the main theorem of this section.

\begin{theorem}
\label{thm:eqd}
Let $d$ be any positive integer. For every graph $G$, it holds that
\[\sigma_{\TT}(G)=\sigma_{\TT^d}(G).\]
\end{theorem}

\begin{proof}
Fix an integer $d\ge 2$ and a graph $G$.
Observe that if $A$ is a coloring $\TT$-base for $G$, then $A\times\TT^{d-1}$ is a coloring $\TT^d$-base for $G$;
this implies that
\[\sigma_{\TT}(G)\le\sigma_{\TT^d}(G).\]
The rest of the proof is devoted to establishing the reverse inequality.
Choose $\varepsilon>0$ arbitrarily and note that by Lemma \ref{lm:eqd}, for all $N\in\NN$ sufficiently large we have 
\begin{equation}
\sigma_{\ZZ_N^d}(G)\ge\sigma_{\TT^d}(G)-\varepsilon.\label{eq:eqd1}
\end{equation}
For any such $N$, since the group $\ZZ_N^d$ is finite, there exists a coloring $\ZZ_N^d$-base $A$ for $G$ with $\sigma_{\ZZ_N^d}(G)\cdot N^d$ elements.
Next choose an integer $k\in\NN$ such that $1-\varepsilon\le\left(\frac{k}{k+2}\right)^d$ and
let $N$ be large enough so that there are at least $d$ distinct primes between $(k+1)N$ and $(k+2)N$ (the Prime Number Theorem implies that this holds for every $N$ sufficiently large);
let $P_1,\ldots,P_d$ be such primes.

Define $A'$ to be the subset of $\ZZ_{P_1}\times\cdots\times\ZZ_{P_d}$ such that
\[A'=\{(x_1+y_1 N,\ldots,x_d+y_d N):(x_1,\ldots,x_d)\in A\mbox{ and }y_1,\ldots,y_d\in [k]\}.\]
Let $f:V(G)\to [N]^d=\ZZ_N^d$ be the function such that $A+f(u)$ and $A+f(v)$ are disjoint for every edge $uv$ of $G$.
Observe that $A'+f(u)$ and $A'+f(v)$ are also disjoint for every edge $uv$ of $G$
when $f(u)$ and $f(v)$ are viewed as elements of $\ZZ_{P_1}\times\cdots\times\ZZ_{P_d}$ (here we use in particular that
each of the primes $P_1,\ldots,P_d$ is at least $(k+1)N$, so that no addition needs to be done modulo $P_i$).
Let $M=P_1P_2\cdots P_d$ and define a mapping $g:\ZZ_M\to\ZZ_{P_1}\times\cdots\times\ZZ_{P_d}$ as
\[g(x)=(x\!\!\mod P_1,\ldots,\; x\!\!\mod P_d).\]
Since $P_1,\ldots,P_d$ are distinct primes,
the mapping $g$ is an isomorphism of the groups $\ZZ_M$ and $\ZZ_{P_1}\times\cdots\times\ZZ_{P_d}$ (by the Chinese Remainder Theorem),
in particular, $g$ is a bijection.

We establish that the set $A''=g^{-1}(A')$ is a coloring $\ZZ_M$-base for $G$.
To do so, consider the function $f'':V(G)\to\ZZ_M$ defined as $f''(v)=g^{-1}(f(v))$, $v\in V(G)$.
Since $g$ is an isomorphism of $\ZZ_M$ and $\ZZ_{P_1}\times\cdots\times\ZZ_{P_d}$,
it holds that $A''+f''(u)$ and $A''+f''(v)$ are disjoint for every edge $uv$ of $G$.
It follows that $A''$ is a coloring $\ZZ_M$-base for $G$,
which implies that
\begin{equation}
\sigma_{\ZZ_M}(G)\ge\frac{|A''|}{M}=\frac{k^d|A|}{M}\ge\frac{k^d|A|}{(k+2)^dN^d}\ge\left(1-\varepsilon\right)\sigma_{\ZZ_N^d}(G).\label{eq:eqd2}
\end{equation}
The inequalities \eqref{eq:eqd1} and \eqref{eq:eqd2} yield that
\[\sigma_{\ZZ_M}(G)\ge \left(1-\varepsilon\right)\left(\sigma_{\TT^d}(G)-\varepsilon\right)\ge\sigma_{\TT^d}(G)-2\varepsilon.\]
Since the choice of $\varepsilon$ was arbitrary,
it follows that $\sigma_{\ZZ_M}(G)$ is at least $\sigma_{\TT^d}(G)$
for every sufficiently large $M$.
Theorem~\ref{thm:eqN} now implies that $\sigma_{\TT}(G)\ge\sigma_{\TT^d}(G)$,
which completes the proof of the theorem.
\end{proof}

Theorem~\ref{thm:eqd} yields the following corollary.

\begin{corollary}
\label{cor:eqZ}
For every graph $G$, it holds that
\[
\sigma_{\TT}(G)=\sup_{Z}\sigma_Z(G)
\]
where the supremum is taken over all finite abelian groups $Z$.
\end{corollary}

\begin{proof}
Fix a graph $G$.
Since every $\ZZ_N$, $N\in\NN$, is a finite abelian group, 
Theorem~\ref{thm:eqN} yields that the supremum given in the statement is at least $\sigma_{\TT}(G)$.
To establish the corollary, we need to show that $\sigma_{\TT}(G)\ge\sigma_Z(G)$ for every finite abelian group $Z$.

Fix a finite abelian group $Z$;
without loss of generality, we may assume that $Z$ is $\ZZ_{M_1}\times\cdots\times\ZZ_{M_d}$.
Let $A$ be a coloring $Z$-base with $\sigma_Z(G)|Z|$ elements and
let $f:V(G)\to Z$ be a function such that $A+f(u)$ and $A+f(v)$
are disjoint for every edge $uv$ of $G$.
Let $N$ be the least common multiple of $M_1,\ldots, M_d$, and
let $\pi$ denote the natural homomorphism $\ZZ_N^d\to Z$,
defined by $(x_1,\ldots,x_d)\mapsto ( x_1 \mod M_1,\ldots,x_d \mod M_d)$.
It is then easy to see that the preimage $A_N:=\pi^{-1}(A)$ is a coloring $\ZZ_N^d$-base,
with corresponding map $f':V\to \ZZ_N^d$ where for each vertex $v$ we let $f'(v)$ be any preimage of $f(v)$ under $\pi$.
We deduce that $\sigma_{\ZZ_N^d}(G)\geq \sigma_{Z}(G)$.
By Lemma~\ref{lm:eqd} and Theorem~\ref{thm:eqd},
we get that $\sigma_{\TT}(G)=\sigma_{\TT^d}(G)\geq \sigma_{\ZZ_N^d}(G)\geq \sigma_{Z}(G)$, and the result follows. 
\end{proof}

We conclude this section with a generalization of Theorem~\ref{thm:eqC} to finite abelian groups.
Recall that if $Z$ is an abelian group and $S$ is a subset of $Z$ such that $S=-S$ and $0\not\in S$,
the \emph{Cayley graph} $C(Z,S)$ is the graph with vertex set $Z$ such that
two vertices $x$ and $y$ of $C(Z,S)$ are adjacent if $y-x\in S$.
Note that a graph $C(N,S)$ defined earlier for an integer $N$ and a subset $S\subseteq\ZZ_N$
is the same as the Cayley graph $C(\ZZ_N,S)$.

\begin{corollary}
\label{cor:eqA}
The gyrochromatic number of a graph $G$ is equal to
\begin{equation}\label{eq:eqA}
\inf_{\substack{Z, S\subseteq Z \\ G\to C(Z,S)}}\chi_f\left(C(Z,S)\right)=
  \inf_{\substack{Z, S\subseteq Z \\ G\to C(Z,S)}}\frac{|Z|}{\alpha\left(C(Z,S)\right)}\;
\end{equation}
where the infimum is taken over all finite abelian groups $Z$.
\end{corollary}

\begin{proof}
The reasoning given in the proof of Theorem~\ref{thm:eqC} yields that the following holds for every abelian group $Z$:
\begin{equation}
\sigma_{Z}(G)=\max_{\substack{S\subseteq Z \\ G\to C(Z,S)}}\frac{\alpha\left(C(Z,S)\right)}{|Z|}.\label{eq:eqA1}
\end{equation}
Since every Cayley graph is vertex-transitive, it holds for every abelian group $Z$ and every $S\subseteq Z$:
\begin{equation}
\chi_f\left(C(Z,S)\right)=\max_{\substack{S\subseteq Z \\ G\to C(Z,S)}}\frac{\alpha\left(C(Z,S)\right)}{|Z|}.\label{eq:eqA2}
\end{equation}
The corollary now follows from \eqref{eq:eqA1}, \eqref{eq:eqA2} and Corollary~\ref{cor:eqZ}.
\end{proof}

We remark that if finite abelian groups $Z$ in the infimum in~\eqref{eq:eqA}
are replaced with all finite groups (with generating set $S$
not containing the identity and satisfying $S=S^{-1}$),
then the infimum is equal to the fractional chromatic number and is always attained.
Indeed, it is well-known that the fractional chromatic number of a graph $G$
is equal to the minimum fractional chromatic number of a Kneser graph that admits a homomorphism from $G$.
Sabidussi's theorem~\cite{Sab64} states that every vertex transitive graph (in particular, every Kneser graph)
is a retract of (hence is homomorphically equivalent to) a Cayley graph, cf.~\cite[Theorem 3.1]{HahT97}.
Hence, the gyrochromatic number can be viewed as a variant of the fractional chromatic number restricted to abelian groups.

\section{Relation to circular and fractional colorings}
\label{sec:sandwich}

We now identify graphs such that their gyrochromatic number
is strictly between their fractional and circular chromatic numbers, and
the difference with the latter number can be arbitrarily large.

\begin{theorem}
\label{thm:sandwich}
There exists a sequence of graphs $(G_k)_{k\in\NN\setminus\{1\}}$ such that
$\chi_f(G_k)<\chi_g(G_k)\le\chi_c(G_k)=k+2$ and
\[\lim_{k\to\infty}\chi_g(G_k)=2.\]
\end{theorem}

\begin{proof}
We set $G_k$ to be the Kneser graph $K(2k^3+k,k^3)$.
Since the circular chromatic number and the chromatic number coincide for Kneser graphs~\cite{ChaLZ13,Che11,LiuZ16},
it follows that the circular chromatic number $\chi_c(G_k)$ is equal to $k+2$.
Recall that the fractional chromatic number of $G_k$ is $\frac{2k^3+k}{k^3}=2+k^{-2}$,
in particular, the limit of the fractional chromatic numbers of $G_k$ is two.

We next show that the gyrochromatic numbers of the graphs $G_k$ converge and
their limit is also two.
Since it holds that $\chi_f(G_k)\le\chi_g(G_k)$ for every graph $G_k$,
the limit (assuming that it exists) must be at least two.
For the upper bound,
consider the Cayley graph $H_k=C(\ZZ_2^{2k^3+k},S)$ with the generating set $S$
consisting of those $x\in\ZZ_2^{2k^3+k}$ that have exactly $2k^3$ entries equal to one.
The graph $H_k$ admits a homomorphism from $G_k$: indeed,
each vertex of $K(2k^3+k,k^3)$ corresponds to a $k^3$-element subset of $[2k^3+k]$ and
we map it to the characteristic vector of this set.
This mapping is indeed a homomorphism from $G_k$ to $H_k$
since two vertices of $G_k=K(2k^3+k,k^3)$ are adjacent if and only if
their corresponding sets are disjoint,
which happens if and only if the difference of their characteristic vectors (modulo two)
has exactly $2k^3$ entries equal to one.
Corollary~\ref{cor:eqA} implies that
\[\limsup_{k\to\infty}\chi_g(G_k)\le\limsup_{k\to\infty}\frac{|H_k|}{\alpha(H_k)}.\]
We next show that the right limit is at most two.

Let $I_k\subseteq V(H_k)$ be the set of those elements of $H_k$ with fewer than $k^3$ entries equal to one.
Since $I_k$ is an independent set in $H_k$ and
it holds that $\binom{n}{m}\le 2^{n}/\Theta(n^{1/2})$ for $0\le m\le n$,
we obtain that
\[\alpha(H_k)\ge |I_k| = \sum_{i=0}^{k^3-1}\binom{2k^3+k}{i}=2^{2k^3+k} \cdot \left( \frac{1}{2} + o(1)\right),\]
which implies that
\[\limsup_{k\to\infty}\frac{|H_k|}{\alpha(H_k)}\le\lim_{k\to\infty}\frac{2^{2k^3+k}}{|I_k|}=2.\]
We conclude that the sequence $\chi_g(G_k)$ converges and its limit is equal to two.

To complete the proof of the theorem, we need to show that $\chi_f(G_k)<\chi_g(G_k)$.
To do so, it suffices to show that $\chi_f(G_k)<\chi_f(G_k\Box G_k)$. Indeed, if we show this, then by Proposition~\ref{prop:prod} and
the fact that the fractional chromatic number of every graph
is at most its gyrochromatic number, we will have $\chi_f(G_k)< \chi_f(G_k\Box G_k)\le\chi_g(G_k\Box G_k)=\chi_g(G_k)$, as required. 

Since the graphs $G_k$ and $G_k\Box G_k$ are vertex-transitive, it holds that
\[\chi_f(G_k)=\frac{|G_k|}{\alpha(G_k)}\qquad\mbox{and}\qquad\chi_f(G_k\Box G_k)=\frac{|G_k\Box G_k|}{\alpha(G_k\Box G_k)}=\frac{|G_k|^2}{\alpha(G_k\Box G_k)}.\]
Hence, it is enough to show that $\alpha(G_k\Box G_k)<|G_k|\alpha(G_k)$.
The Erd\H{o}s-Ko-Rado Theorem yields that $\alpha(G_k)=\binom{2k^3+k-1}{k^3-1}$ and
that for every independent set of $G_k$ of this size, there exists $x\in [2k^3+k]$ such that
the vertices of the independent set
correspond to the $\binom{2k^3+k-1}{k^3-1}$ $k^3$-element subsets of $[2k^3+k]$ containing $x$.
In particular,
any two independent sets of the maximum cardinality in $G_k$
have a non-empty intersection (here we use that $k\ge 2$).

Suppose that the graph $G_k\Box G_k$ has an independent set of size $|G_k|\alpha(G_k)$, and
let $I$ be such an independent set.
For every vertex $v$ of $G_k$, the set $I_v=\{w:(v,w)\in I\}$ is an independent set in $G_k$ (by definition of the Cartesian product).
Since the set $I$ contains $|G_k|\alpha(G_k)$ elements and $|I_v|\le\alpha(G_k)$ for every vertex $v\in V(G_k)$,
we obtain that $|I_v|=\alpha(G_k)$ for every $v\in V(G_k)$.
Let $v$ and $v'$ be two adjacent vertices of $G_k$.
It follows from the previous paragraph that $I_v$ and $I_v'$ have a common vertex $w$, i.e., the set $I$ contains both $(v,w)$ and $(v',w)$,
which contradicts that $I$ is an independent set.
Hence, the graph $G_k\Box G_k$ has no independent set of size $|G_k|\alpha(G_k)$,
i.e. $\alpha(G_k\Box G_k)<|G_k|\alpha(G_k)$.
This finishes the proof of the theorem.
\end{proof}

Theorem~\ref{thm:sandwich} immediately implies that
the gap between the gyrochromatic number and the circular chromatic number
can be arbitrarily large and
there exists a graph such that
its gyrochromatic number is strictly between its fractional and its circular chromatic numbers.

\begin{corollary}
\label{cor:sandwich}
For every $k\in\NN$, there exists a graph $G$ such that $\chi_f(G)<\chi_g(G)\le\chi_c(G)-k$.
\end{corollary}

\section{Existence of optimal gyrocoloring}
\label{sec:attain}

In this section, we establish the existence of a graph $G$ such that
there is no coloring $\TT$-base for $G$ of measure $\sigma_{\TT}(G)$,
i.e., the supremum in \eqref{eq:sigma} is not attained.
In other words, there exists a graph $G$ with no $\chi_g(G)$-gyrocoloring.

Let $G_5$ be the graph with vertex set $\ZZ_5^2$ and
two vertices $(i,j)$ and $(i',j')$ adjacent if $i'-i\in\{2,3\}$ or $j'-j\in\{2,3\}$ (calculations modulo five). In other words, $G_5$ is the Cayley graph on $\ZZ_5^2$ with generating set $\ZZ_5^2\setminus \{-1,0,1\}^2$. Proposition~\ref{prop:attain2} and Theorem~\ref{thm:attain}, which we will prove in this section,
yield that $G_5$ has no coloring $\TT$-base with measure $\sigma_{\TT}(G_5)$.

We begin by analyzing the structure of independent sets of $G_5$.

\begin{proposition}
\label{prop:attain1}
The independence number of $G_5$ is four and
the only independent sets of size four are the following 25 sets:
\[I_v=\{v,v+(0,1),v+(1,0),v+(1,1)\}\]
where $v\in\ZZ_5^2$.
\end{proposition}

\begin{proof}
Let $X$ be an independent set in $G_5$ and let $(i,j)\in\ZZ_5^2$ be any vertex of $G_5$ contained in $X$.
Observe that the set $X$ contains in addition to the vertex $(i,j)$
at most one of the vertices $(i-1,j)$ and $(i+1,j)$,
at most one of the vertices $(i,j-1)$ and $(i,j+1)$,
at most one of the vertices $(i-1,j-1)$, $(i-1,j+1)$, $(i+1,j-1)$ and $(i+1,j+1)$, and
no other vertex.
In particular, the set $X$ has size at most four.
If the set $X$ has size four, we can assume by symmetry that
$X$ contains the vertex $(i+1,j)$ from the first pair and the vertex $(i,j+1)$ from the second pair.
If $X$ contains the vertices $(i,j)$, $(i+1,j)$ and $(i,j+1)$,
then the only vertex from the last quadruple that it can contain is $(i+1,j+1)$.
The statement of the proposition now follows.
\end{proof}

We next compute $\sigma_{\TT}(G_5)$.
We remark that the construction of the coloring $\ZZ_5^2$-base for $G_5$ presented in the proof of the next proposition
yields a coloring $\TT^2$-base with measure $4/25$.

\begin{proposition}
\label{prop:attain2}
It holds that $\sigma_{\TT}(G_5)=4/25$.
\end{proposition}

\begin{proof}
Since $\alpha(G_5)=4$ and $G_5$ is vertex-transitive, it follows that $\chi_f(G_5)=25/4$,
which implies that $\sigma_{\TT}(G_5)\le 4/25$.
On the other hand, the set $A=\{(0,0),(0,1),(1,0),(1,1)\}$ is a coloring $\ZZ_5^2$-base for $G_5$:
indeed, the sets $A+v$ and $A+w$ are disjoint for every edge $vw$ of $G_5$.
Hence, $\sigma_{\ZZ_5^2}(G_5)\ge 4/25$, which yields that $\sigma_{\TT}(G_5)\ge 4/25$ by Corollary~\ref{cor:eqZ}.
\end{proof}

\noindent We next show that every coloring $\TT$-base of $G_5$ of measure $4/25$
must have a very particular structure.
In what follows, we will write $X\cong Y$ if two sets $X$ and $Y$ differ on a null set.

\begin{lemma}
\label{lm:weights}
Let $A\subseteq\TT$ be a coloring $\TT$-base of the graph $G_5$ and
let $f:V(G_5)\to\TT$ be such that $A+f(v)$ and $A+f(w)$ are disjoint for every edge $vw$;
let $B_v$, $v\in V(G_5)$ be the set $A+f(v)$.
If the measure of $A$ is $4/25$,
then there exist disjoint measurable subsets $C_{i,j}$, $(i,j)\in\ZZ_5^2$, such that
\[B_{i,j}\cong C_{i,j}\cup C_{i+1,j}\cup C_{i,j+1}\cup C_{i+1,j+1}\]
for every $(i,j)\in\ZZ_5^2$, and
the measure of each set $C_{i,j}$, $(i,j)\in\ZZ_5^2$, is $1/25$.
\end{lemma}

\begin{proof}
Let $\II$ be the set containing all independent sets of vertices of $G_5$.
For each $I\in\II$, we define the following measurable subset of $\TT$:
\[
D_I= \{x\in\TT: x\in A+f(v)\textrm{ if and only if } v\in I\},
\]
i.e., $D_I$ contains those $x\in\TT$ that
are contained in the sets $A+f(v)$ for $v\in I$ and
are not contained in the sets $A+f(v)$ for $v\not\in I$.
Observe that every $x\in \TT$ belongs to $D_I$ for some $I\in\II$:
indeed, the set containing the vertices $v$ such that $x\in A+f(v)$ is independent (note that $\emptyset \in\II$).
Hence, the sets $D_I$, $I\in\II$, partition $\TT$. We also have from the definition of $D_I$ that
\begin{equation}\label{eq:basepart}
A+f(v)=\bigcup_{I\in\II: I\ni v} D_I,
\end{equation}
and it follows that $\sum_{I\in\II}|I|\cdot\mu(D_I)=\sum_{v\in V(G_5)}\mu(A+f(v))$. On the other hand, since the measure of $A$ is $4/25$, we have $\sum_{v\in V(G_5)}\mu(A+f(v))=4$. Hence, using Proposition \ref{prop:attain1}, we have
\[
4 = \sum_{I\in\II}|I|\cdot\mu(D_I) = 4 \sum_{I\in\II:|I|=4} \mu(D_I) + \sum_{I\in\II:|I|< 4} |I|\cdot \mu(D_I).
\]
This implies that $\mu(D_I)=0$ for all $I\in\II$ with $|I|<4$ (using $\sum_{I\in\II}\mu(D_I) = 1$). Setting \[C_{i,j}=D_{\{(i-1,j-1),(i-1,j),(i,j-1),(i,j)\}},\]
we have by \eqref{eq:basepart} that $B_{i,j} \cong C_{i,j}\cup C_{i+1,j}\cup C_{i,j+1}\cup C_{i+1,j+1}$ for every $(i,j)\in\ZZ_5^2$.
To complete the proof, we need to show each set $C_{i,j}$ has measure $1/25$. 

Recalling the notation $I_v$ from Proposition \ref{prop:attain1},
let $M$ be the matrix with rows and columns indexed by the elements of $\ZZ_5^2$ such that
$M_{(i,j),(i',j')}$ is one if $(i,j)\in I_{(i',j')}$ and zero otherwise.
Further, let $x$ be the vector with entries indexed by the elements of $\ZZ_5^2$ such that
$x_{(i',j')}$ is the measure of $D_{I_{(i',j')}}$.
Observe (using \eqref{eq:basepart} and that $\mu(D_I)=0$ for $|I|<4$) that $Mx$ is the vector with all entries equal to $4/25$.
We next show that the matrix $M$ is invertible.
This would imply that the vector $x$ with all entries equal to $1/25$ is the only vector such that
$Mx$ is the vector with all entries equal to $4/25$, which would complete the proof.

Assume that the matrix $M$ is singular,
i.e., there exists a non-zero vector $x$ such that $Mx$ is the zero vector.
The entries of $x$ can be interpreted as numbers on the toroidal $5\times 5$ grid such that
each of the $25$ quadruples of entries forming a square sums to zero;
a ``square" stands here for a translate of $\{(0,0),(1,0),(0,1),(1,1)\}$ in the grid.

We can assume that the first row of the grid is $\alpha,\beta_1,\ldots,\beta_4$ and
the first column is $\alpha,\gamma_1,\ldots,\gamma_4$,
which yields that the numbers are assigned to the grid as follows.
\[
\begin{array}{ccccc}
\alpha & \beta_1 & \beta_2 & \beta_3 & \beta_4 \\
\gamma_1 & -\alpha-\beta_1-\gamma_1 & +\alpha-\beta_2+\gamma_1 & -\alpha-\beta_3-\gamma_1 & +\alpha-\beta_4+\gamma_1 \\
\gamma_2 & +\alpha+\beta_1-\gamma_2 & -\alpha+\beta_2+\gamma_2 & +\alpha+\beta_3-\gamma_2 & -\alpha+\beta_4+\gamma_2 \\
\gamma_3 & -\alpha-\beta_1-\gamma_3 & +\alpha-\beta_2+\gamma_3 & -\alpha-\beta_3-\gamma_3 & +\alpha-\beta_4+\gamma_3 \\
\gamma_4 & +\alpha+\beta_1-\gamma_4 & -\alpha+\beta_2+\gamma_4 & +\alpha+\beta_3-\gamma_4 & -\alpha+\beta_4+\gamma_4 \\
\end{array}
\]
Since there is a square containing $\beta_1$, $\beta_2$, $\alpha+\beta_1-\gamma_4$ and $-\alpha+\beta_2+\gamma_4$,
we obtain $\beta_1=-\beta_2$.
By considering other squares wrapping around,
we obtain that $\beta_1=-\beta_2=\beta_3=-\beta_4$,
$\gamma_1=-\gamma_2=\gamma_3=-\gamma_4$ and $\beta_4=-\gamma_4$.
Hence, the table can be rewritten as follows.
\[
\begin{array}{ccccc}
\alpha & \beta_1 & -\beta_1 & \beta_1 & -\beta_1 \\
-\beta_1 & -\alpha & +\alpha & -\alpha & +\alpha \\
+\beta_1 & +\alpha & -\alpha & +\alpha & -\alpha \\
-\beta_1 & -\alpha & +\alpha & -\alpha & +\alpha \\
+\beta_1 & +\alpha & -\alpha & +\alpha & -\alpha
\end{array}
\]
Considering two of the squares containing the entry $\alpha$ in the top left corner,
we obtain $2(\alpha+\beta_1)=0$ and $2(\alpha-\beta_1)=0$. Hence $\alpha=\beta_1=0$, i.e., the vector $x$ is zero.
We conclude that the matrix $M$ is invertible, which completes the proof.
\end{proof}

We are now ready to prove the main lemma of this section.
To simplify our notation,
we will understand the subscripts indexing sets $B_{i,j}$ and $C_{i,j}$ in Lemma~\ref{lm:weights} as pairs $(i,j)\in\ZZ_5^2$,
which allows us to perform addition as with the elements of $\ZZ_5^2$,
e.g., $C_{(0,1)+(1,2)}$ is the set $C_{(1,3)}=C_{1,3}$.
In addition, we write $t_{v\to w}$ for $f(w)-f(v)$,
when $A$ is a coloring $\TT$-base and $f:V(G)\to\TT$ is a function such that $A+f(v)$ and $A+f(w)$ are disjoint for every edge $vw$.

\begin{lemma}
\label{lm:cellshifts}
Let $A\subseteq\TT$ be a coloring $\TT$-base of the graph $G_5$ with measure $4/25$ and
let $C_{i,j}\subseteq\TT$, $(i,j)\in\ZZ_5^2$, be the sets as in Lemma~\ref{lm:weights}.
For every $v\in\ZZ_5^2$ and $w\in\{(0,0),(0,1),(1,0),(1,1)\}$, it holds that
\begin{align}
C_{v+w}+ t_{v\to v+(1,0)} & \cong C_{v+w+(1,0)} \label{eq:cellshifts-h}\\
C_{v+w}+ t_{v\to v+(0,1)} & \cong C_{v+w+(0,1)} \label{eq:cellshifts-v}.
\end{align}
\end{lemma}

\begin{figure}
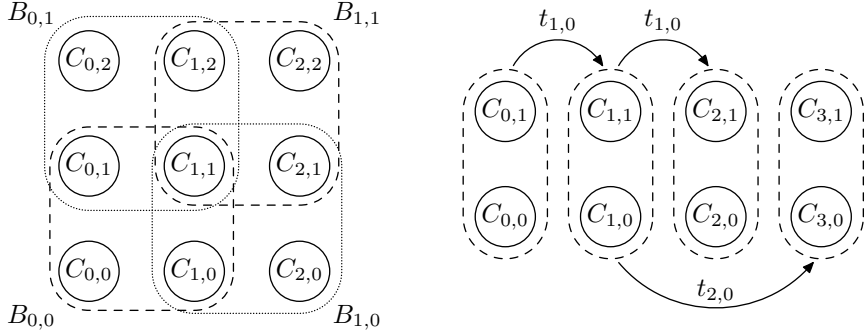

\begin{center}
\epsfbox{gyrocol-1.mps}
\hskip 5ex
\epsfbox{gyrocol-2.mps}
\end{center}
\caption{Visualization of the notation used in the proof of Lemma~\ref{lm:cellshifts} and
         the equalities \eqref{eq:cellshifts-h1}, \eqref{eq:cellshifts-h2} and \eqref{eq:cellshiftbytwo}.}
\label{fig:cellshifts}
\end{figure}

\begin{proof}
By symmetry, we will assume that $v=(0,0)$ in our presentation
with the exception of equality~\eqref{eq:cellshiftbytwo-v}, which we formulate in the general setting.
Throughout the proof, we write $B_{i,j}\subseteq\TT$ for the sets as in Lemma~\ref{lm:weights};
to simplify our notation, we also write $t_{i,j}$ for $t_{(0,0)\to (i,j)}$.

Our first goal is to prove the following weaker statement, which is also visualized in Figure~\ref{fig:cellshifts}.
\begin{align}
(C_{0,0}\cup C_{0,1})+t_{1,0} & \cong C_{1,0}\cup C_{1,1} \label{eq:cellshifts-h1}\\
(C_{1,0}\cup C_{1,1})+t_{1,0} & \cong C_{2,0}\cup C_{2,1} \label{eq:cellshifts-h2}
\end{align}
Since the vertices $(0,0)$ and $(2,0)$ are adjacent, the sets $B_{0,0}$ and $B_{2,0}=B_{0,0}+t_{2,0}$ are disjoint, and
so are these sets shifted by $t_{1,0}$,
i.e., the sets $B_{0,0}+t_{1,0}=B_{1,0}$ and $B_{0,0}+t_{2,0}+t_{1,0}=B_{1,0}+t_{2,0}$.
In particular,
the intersection of $C_{2,0}\cup C_{2,1}\subseteq B_{1,0}$ and
$(C_{1,0}\cup C_{1,1})+t_{2,0}\subseteq B_{1,0}+t_{2,0}$ is empty.
Since both $C_{2,0}\cup C_{2,1}$ and $(C_{1,0}\cup C_{1,1})+t_{2,0}$ are subsets of $B_{2,0}$ and
the measure of each of them is half of the measure of $B_{2,0}$,
it follows that the sets $(C_{1,0}\cup C_{1,1})+t_{2,0}$ and
$B_{2,0}\setminus(C_{2,0}\cup C_{2,1})=C_{3,0}\cup C_{3,1}$ are the same (up to a null set), i.e.,
\begin{equation}
(C_{1,0}\cup C_{1,1})+t_{2,0}\cong C_{3,0}\cup C_{3,1}. \label{eq:cellshiftbytwo}
\end{equation}
We formulate \eqref{eq:cellshiftbytwo} for an arbitrary vertex $v$ since we need the statement later:
\begin{equation}
(C_{v+(1,0)}\cup C_{v+(1,1)})+t_{v\to v+(2,0)}\cong C_{v+(3,0)}\cup C_{v+(3,1)}. \label{eq:cellshiftbytwo-v}
\end{equation}
We next apply \eqref{eq:cellshiftbytwo-v} with $v=(1,0)$ and $v=(3,0)$ as follows:
\begin{align*}
(C_{2,0}\cup C_{2,1})-t_{(0,0)\to(1,0)} & =  (C_{2,0}\cup C_{2,1})+t_{(1,0)\to(0,0)} \\
& = (C_{2,0}\cup C_{2,1})+t_{(1,0)\to(3,0)}+ t_{(3,0)\to (0,0)}\\
& \cong ( C_{4,0}\cup C_{4,1})+ t_{(3,0)\to (0,0)} \\
& \cong C_{1,0}\cup C_{1,1},
\end{align*}
which proves \eqref{eq:cellshifts-h2}. Since both the sets $(C_{0,0}\cup C_{0,1}\cup C_{1,0}\cup C_{1,1})+t_{1,0}$ and
$C_{1,0}\cup C_{1,1}\cup C_{2,0}\cup C_{2,1}$ are equal to $B_{1,0}$,
in particular, they are the same set, and all sets $C_{i,j}$, $(i,j)\in\ZZ_5^2$, are disjoint,
the equality \eqref{eq:cellshifts-h1} also follows.

An argument symmetric to the one used to prove \eqref{eq:cellshifts-h1} and \eqref{eq:cellshifts-h2}
yields the following.
\begin{align}
(C_{0,0}\cup C_{1,0})+t_{0,1} & \cong C_{0,1}\cup C_{1,1} \label{eq:cellshifts-v1}\\
(C_{0,1}\cup C_{1,1})+t_{0,1} & \cong C_{0,2}\cup C_{1,2} \label{eq:cellshifts-v2}
\end{align}

Next suppose for a contradiction that the intersection of $C_{0,0}+t_{1,0}$ and $C_{1,1}$ has positive measure, and
let $X$ be the set $C_{0,0}\cap(C_{1,1}-t_{1,0})$.
The equality \eqref{eq:cellshifts-v2} implies that $X+t_{1,0}+t_{0,1}$ is a subset of $C_{0,2}\cup C_{1,2}$ (up to a null set).
On the other hand $X+t_{0,1}\subseteq C_{0,0}+t_{0,1}$ is a subset of $C_{0,1}\cup C_{1,1}$ by \eqref{eq:cellshifts-v1}, and this is a subset of $B_{0,0}$, 
hence $(X+t_{0,1})+t_{1,0}\subseteq B_{0,0}+t_{1,0}=B_{1,0}$. However $B_{1,0}\cong C_{1,0}\cup C_{1,1}\cup C_{2,0}\cup C_{2,1}$ has null intersection with $C_{0,2}\cup C_{1,2}$. We have thus deduced that $X+t_{1,0}+t_{0,1}$ is included (up to a null set) in the null set $B_{1,0}\cap (C_{0,2}\cup C_{1,2})$, which contradicts the assumption that $X+t_{0,1}+t_{1,0}$ has positive measure.
We conclude that the intersection of $C_{0,0}+t_{1,0}$ and $C_{1,1}$ is null.
Since all the sets $C_{0,0}$, $C_{0,1}$, $C_{1,0}$ and $C_{1,1}$ have the same measure,
the equality \eqref{eq:cellshifts-h1} implies that
\begin{equation}
C_{0,0}+t_{1,0} \cong C_{1,0}\mbox{ and }C_{0,1}+t_{1,0} \cong C_{1,1}.\label{eq:celshifts-heq1}
\end{equation}
A symmetric argument implies that
\begin{equation}
C_{0,0}+t_{0,1} \cong C_{0,1}\mbox{ and }C_{1,0}+t_{0,1} \cong C_{1,1}.\label{eq:celshifts-veq1}
\end{equation}

We now prove that the intersection of the sets $C_{1,1}+t_{0,1}$ and $C_{0,2}$ is null.
Assume the contrary, i.e., the set $X=C_{1,1}\cap (C_{0,2}-t_{0,1})$ has positive measure.
The equality \eqref{eq:cellshiftbytwo-v} applied with $v=(0,1)$ implies that
\[(C_{1,1}\cup C_{1,2})+t_{(0,1)\to (2,1)}\cong C_{3,1}\cup C_{3,2}.\]
Since it holds that $B_{0,1}+t_{(0,1)\to (2,1)}=B_{2,1}$,
we get that
\[(C_{0,1}\cup C_{1,1}\cup C_{0,2}\cup C_{1,2})+t_{(0,1)\to (2,1)}\cong C_{2,1}\cup C_{3,1}\cup C_{2,2}\cup C_{3,2}.\]
Since all the sets $C_{i,j}$ are disjoint,
it follows that
\[(C_{0,1}\cup C_{0,2})+t_{(0,1)\to (2,1)}\cong C_{2,1}\cup C_{2,2}.\]
Since $X+t_{0,1}$ is a subset of $C_{0,2}$, we obtain that
\begin{equation}
X+t_{2,1}=X+t_{0,1}+t_{(0,1)\to (2,1)}\subseteq C_{0,2}+t_{(0,1)\to (2,1)}\subseteq C_{2,1}\cup C_{2,2}\subseteq B_{1,1}.\label{eq:cellshifts-hh}
\end{equation}
On the other hand, it holds that $X\subseteq C_{1,1}\subseteq B_{1,1}$.
Hence, the intersection of the sets $B_{1,1}$ and $B_{1,1}+t_{2,1}$ contains the set $X+t_{2,1}$,
in particular, it has positive measure.
However, this is impossible because $B_{1,1}=B_{0,0}+t_{1,1}$ and $B_{1,1}+t_{2,1}=B_{2,1}+t_{1,1}$ and
the sets $B_{0,0}$ and $B_{2,1}$ are disjoint.
We conclude that the intersection of sets $C_{1,1}+t_{0,1}$ and $C_{0,2}$ is null.
Using \eqref{eq:cellshifts-v2}, we obtain that
\[C_{0,1}+t_{0,1} \cong C_{0,2}\mbox{ and }C_{1,1}+t_{0,1} \cong C_{1,2},\]
and a symmetric argument yields that
\[C_{1,0}+t_{1,0} \cong C_{2,0}\mbox{ and }C_{1,1}+t_{1,0} \cong C_{2,1}.\]
The proof is now complete.
\end{proof}

Our next step is to deduce that the elements $t_{v\to w}$ can be replaced by integer combinations of the elements $t_{(0,0)\to (1,0)}$ and $t_{(0,0)\to (0,1)}$.

\begin{lemma}
\label{lm:2generators}
Let $A\subseteq\TT$ be a coloring $\TT$-base of the graph $G_5$ with measure $4/25$ and
let $C_{i,j}\subseteq\TT$, $(i,j)\in\ZZ_5^2$, be the sets as in Lemma~\ref{lm:weights}.
It holds that
\[C_{i\!\!\!\!\mod 5,\;\,j\!\!\!\!\mod 5}\cong C_{0,0}+it_{(0,0)\to (1,0)}+jt_{(0,0)\to (0,1)}\]
for any two non-negative integers $i$ and $j$.
\end{lemma}

\begin{proof}
We proceed by induction on $i+j\in \ZZ$;
all calculations with subscripts are done modulo five throughout the proof.
The base of the induction is the case $i+j\in\{0,1\}$, which is implied by Lemma~\ref{lm:cellshifts}.
For the rest of the proof fix $i$ and $j$.
By symmetry, we may assume that $j>0$.
Applying Lemma~\ref{lm:cellshifts} with $v=(0,0),(1,0),\ldots,(i,0),(i,1),\ldots,(i,j-1)$,
we obtain the following.
\begin{align}
C_{i,j-1} &\cong C_{0,0}+t_{(0,0)\to(1,0)}+\cdots+t_{(i-1,0)\to (i,0)}+t_{(i,0)\to (i,1)}+\cdots+t_{(i,j-2)\to (i,j-1)}\label{eq:2gen1}\\
C_{i,j} &\cong C_{0,1}+t_{(0,0)\to(1,0)}+\cdots+t_{(i-1,0)\to (i,0)}+t_{(i,0)\to (i,1)}+\cdots+t_{(i,j-2)\to (i,j-1)}\label{eq:2gen2}
\end{align}
Since $C_{0,1}\cong C_{0,0}+t_{(0,0)\to (0,1)}$ by Lemma~\ref{lm:cellshifts},
we conclude using \eqref{eq:2gen1} and \eqref{eq:2gen2} that
\begin{equation}
C_{i,j}\cong C_{i,j-1}+t_{(0,0)\to (0,1)}.\label{eq:2gen3}
\end{equation}
On the other hand, the induction yields
\[C_{i,j-1} \cong C_{0,0}+it_{(0,0)\to (1,0)}+(j-1)t_{(0,0)\to (0,1)},\]
which combines with \eqref{eq:2gen3} to imply
\[C_{i,j} \cong  C_{0,0}+it_{(0,0)\to (1,0)}+jt_{(0,0)\to (0,1)}.\]
\end{proof}

We are now ready to prove the main theorem of this section.

\begin{theorem}
\label{thm:attain}
The graph $G_5$ has no coloring $\TT$-base with measure $4/25$.
\end{theorem}

\begin{proof}
Suppose that there exists a coloring $\TT$-base $A\subseteq\TT$ with measure $4/25$, and
let $C_{i,j}\subseteq\TT$, $(i,j)\in\ZZ_5^2$, be the sets as in Lemma~\ref{lm:weights}.
Further, let $\tau=t_{(0,0)\to(1,0)}$ and $\tau'=t_{(0,0)\to(0,1)}$;
note that $\tau$ and $\tau'$ do not necessarily generate a subgroup isomorphic to $\ZZ_5^2$.
By Lemma~\ref{lm:2generators} we have $C_{0,0}\cong C_{0,0}+5\tau n$ and $C_{0,0}\cong C_{0,0}+5\tau' n$ for any integer $n$.
If $5\tau$ were irrational, then the measure of $C_{0,0}$ would be either zero or one.
Therefore, $5\tau $ is rational.
Similarly, $5\tau'$ is rational.
It follows that both $\tau$ and $\tau'$ are rational.
Let $p,q,r,p',q',r'$ be non-negative integers such that
\[\tau=\frac{p}{5^rq}\qquad\mbox{and}\qquad \tau'=\frac{p'}{5^{r'}q'},\]
$p$ and $5^rq$ are coprime,
$p'$ and $5^{r'}q'$ are coprime, and
neither $q$ nor $q'$ is divisible by five.
By symmetry, we may assume that $r\le r'$.
Let $k$ be an integer such that $5^{r'-r}p$ and $kp'$ are congruent modulo $5^{r'}$;
note that such $k$ exists since $p'$ and $5^{r'}$ are coprime.
Next observe that
\begin{equation}
q\tau=\frac{5^{r'-r}p}{5^{r'}}\qquad\mbox{and}\qquad kq'\tau'=\frac{kp'}{5^{r'}}=\frac{5^{r'-r}p}{5^{r'}}\;\mod 1.
\label{eq:att1}
\end{equation}
By Lemma~\ref{lm:2generators}, we obtain that
\begin{align*}
C_{0,0}+q\tau &\cong C_{0,0}+qt_{(0,0)\to(1,0)}\cong C_{q\!\!\!\mod 5,\;\,0}\\
C_{0,0}+kq'\tau' &\cong C_{0,0}+kq't_{(0,0)\to(0,1)}\cong C_{0,\;\,kq'\!\!\!\mod 5}.
\end{align*}
Since it holds that $q\!\!\mod 5\not=0$ and the sets $C_{i,j}$ are disjoint sets of measure $1/25$ by Lemma~\ref{lm:weights},
we obtain that
the intersection of the sets $C_{0,0}+q\tau\cong C_{q\!\!\mod 5,\;\,0}$ and $C_{0,0}+kq'\tau'\cong C_{0,\;\,kq'\!\!\mod 5}$ is null.
However, $q\tau$ and $kq'\tau'$ is the same element of $\TT$ by \eqref{eq:att1}, i.e., $C_{0,0}+q\tau\cong C_{0,0}+kq'\tau'$.
This contradicts our assumption on the existence of a coloring $\TT$-base with measure $4/25$.
\end{proof}

\section{Conclusion}

We finish with giving four open problems that we find particularly interesting and
briefly mentioning a relation of the gyrochromatic number to another graph parameter,
the ultimate independence ratio of a graph.
The \emph{independence ratio} $i(G)$ of a graph $G$ is the ratio $\alpha(G)/|V(G)|$;
the \emph{ultimate independence ratio} $I(G)$,
which was introduced in~\cite{HelYZ94},
is the limit of the independence ratios of Cartesian powers of $G$:
\[I(G)=\lim_{k\to\infty}\frac{\alpha(G^k)}{|V(G^k)|},\]
where $G^k$ is the Cartesian product of $k$ copies of $G$.
The inverse of this quantity is the \emph{ultimate fractional chromatic number} $\chi_F(G)$ of a graph $G$ and
the following holds~\cite{HahHP95,Zhu96}:
\[\chi_F(G)=\frac{1}{I(G)}=\lim_{k\to\infty}\chi_f(G^k).\]
Zhu~\cite[p.\ 236]{Zhu96} related the ultimate fractional chromatic number to coloring bases of abelian groups (though he used different terminology), via the following inequality:
\[
I(G)\geq \sup\left\{\tfrac{\alpha(H)}{|V(H)|}\: \big|\: G\to H, H \text{ is a Cayley graph on a finite abelian group}\right\}.
\]
Using Corollary~\ref{cor:eqA}, we conclude that $\chi_F(G)\le\chi_g(G)$.
Hence, we obtain that the following holds for every graph $G$:
\[\chi_f(G)\le\chi_F(G)\le\chi_g(G)\le\chi_c(G)\le\chi(G).\]
It seems plausible that $\chi_F(G)$ and $\chi_g(G)$
differ for some graphs $G$, and it would be interesting to give examples of such graphs.

\begin{problem}
\label{prob:0}
Construct a (connected) graph $G$ such that $\chi_F(G)<\chi_g(G)$.
\end{problem}

We finish with two problems on the gyrochromatic number and its relation to fractional and circular chromatic numbers,
which we believe to be of particular interest.

\begin{problem}
\label{prob:1}
Does there exist a function $f:\RR\to\RR$ such that $\chi_g(G)\le f(\chi_f(G))$ for every graph $G$?
\end{problem}

\begin{problem}
\label{prob:2}
Does there exist a finite graph $G$ such that the gyrochromatic number of $G$ is not rational?
\end{problem}

Observe that a function $f$ in Problem~\ref{prob:1} exists for all graphs $G$ if and only if it exists for Kneser graphs,
i.e., such a function $f$ exists if and only if the gyrochromatic number of Kneser graph $K(m,n)$ is at most $f(m/n)$.
Also note that Theorem~\ref{thm:sandwich} implies that the circular chromatic number of a graph
cannot be upper bounded by a function of its gyrochromatic number,
i.e., a function $f$ as in Problem~\ref{prob:1} does not exist for the gyrochromatic and the circular chromatic number.

In Section~\ref{sec:attain}, we have constructed a graph $G$ such that
there is no coloring $\TT$-base for $G$ of measure $\sigma_{\TT}(G)$,
i.e., the supremum in \eqref{eq:sigma} is not attained.
However, the constructed graph $G$
has a coloring $\TT^2$-base with measure $\chi_g(G)^{-1}$ (see Proposition~\ref{prop:attain2} and the remark before it),
which leads to the following problem.

\begin{problem}
\label{prob:3}
Does there exist for every graph $G$ an integer $d$ such that
$G$ has a $\TT^d$-coloring base with measure $\chi_g(G)^{-1}$?
\end{problem}

\section*{Acknowledgements}

The authors would like to thank Xuding Zhu and two anonymous reviewers for their valuable comments.

\bibliographystyle{abbrv}
\bibliography{gyrocol}

\Addresses

\end{document}